\documentclass{proc-l}

\usepackage{enumitem} 

\newtheorem{theorem}{Theorem}[section]
\newtheorem{lemma}[theorem]{Lemma}

\theoremstyle{definition}
\newtheorem{definition}[theorem]{Definition}
\newtheorem{example}[theorem]{Example}
\newtheorem{corollary}[theorem]{Corollary}

\theoremstyle{remark}
\newtheorem{remark}[theorem]{Remark}

\numberwithin{equation}{section}

\begin{document}

\title{Sign of the solutions of linear fractional differential equations and some applications}

\author{Rui A. C. Ferreira}
\address{Grupo F\'isica-Matem\'atica, Faculdade de Ci\^encias, Universidade de Lisboa,
Av. Prof. Gama Pinto, 2, 1649-003, Lisboa, Portugal.}
%\curraddr{Department of Mathematics and Statistics,
%Case Western Reserve University, Cleveland, Ohio 43403}
\email{raferreira@fc.ul.pt}
\thanks{The author was supported by the ``Funda\c{c}\~{a}o para a Ci\^encia e a Tecnologia (FCT)" through the program ``Stimulus of Scientific Employment, Individual Support-2017 Call" with reference CEECIND/00640/2017.}

%\author{Author Two}
%\address{Mathematical Research Section, School of Mathematical Sciences,
%Australian National University, Canberra ACT 2601, Australia}

%    General info
\subjclass[2000]{Primary 26A33, 34A30; Secondary 49K30}

%\date{January 1, 2001 and, in revised form, June 22, 2001.}


\keywords{Linear equation, fractional derivative, Herglotz variational problem}

\begin{abstract}
In this work we wish to highlight some consequences of a recent result proved in [N. D. Cong\ and\ H. T. Tuan, Generation of nonlocal fractional dynamical systems by fractional differential equations, J. Integral Equations Appl. {\bf 29} (2017), no.~4, 585--608]. Particular emphasis will be given to its application on fractional variational problems of Herglotz type.
\end{abstract}

\maketitle

\section{Preamble}

Very recently, in 2017, Cong \emph{et. al.} \cite{Cong} proved the following result (we refer the reader to Section \ref{sect2} for the definitions appearing below):

\begin{theorem}\label{CongThm}
Consider the following fractional differential equation
\begin{equation}\label{CapDE}
    ^CD_{a^+}^{\alpha} [x](t)=f(t,x(t)),\quad 0<\alpha\leq 1,
\end{equation}
where $f:[a,\infty)\times\mathbb{R}\to\mathbb{R}$ is a continuous function satisfying the Lipschitz condition
\begin{equation}\label{Lip}
    |f(t,x)-f(t,y)|\leq L(t)|x-y|,\quad t\in [a,\infty),\ x,y\in\mathbb{R},\ L\in C([a,\infty),\mathbb{R}).
\end{equation}
Then, for any two different initial values $x_{1a}\neq x_{2a}$ in $\mathbb{R}$, the solutions $x_1$ and $x_2$ of \eqref{CapDE} starting from $x_{1a} = x_1(a)$ and $x_{2a} = x_2(a)$ verify $x_1(t)\neq x_2(t)$ for all $t\in[a,\infty)$.
\end{theorem}

\begin{remark}
The existence and uniqueness of (continuous) solutions for \eqref{CapDE} with the given initial conditions is guaranteed by, e.g., \cite[Theorem 2]{Baleanu}.
\end{remark}

\begin{remark}
Theorem \ref{CongThm} was conjectured in 2008 by Diethelm \cite{Diethelm1} and solved partially therein. However, it was only in 2017 that a complete and correct proof of it was given (see \cite{Cong} for the historical developments regarding this result).
\end{remark}
In particular, for $0<\alpha<1$ consider the linear initial value problem (IVP)
\begin{equation}\label{IVPCap}
^CD_{a^+}^{\alpha} [x](t)=g(t)x(t),\quad x(a)=x_a>0,
\end{equation}
for which $f(t,x)=g(t)x$ (with $g\in C([a,\infty),\mathbb{R})$) obviosuly satisfies \eqref{Lip}. Then we may conclude from Theorem \ref{CongThm} that the solution of \eqref{IVPCap} is positive on $[a,\infty)$. We find this result of much interest and, to the best of our knowledge, it was not sufficiently highlighted in the literature so far; therefore, we shall write it in the following:

\begin{theorem}\label{thm1.3}
Let $0<\alpha< 1$ and $g\in C([a,\infty),\mathbb{R})$. Then the solution of the IVP \eqref{IVPCap} is positive on $[a,\infty)$.
\end{theorem}

We may extract several interesting consequences of Theorem \ref{thm1.3} and we will refer some in Section \ref{ApplicSect} of this work. But before we go into it we want to observe that, following the same steps as those of the proof of Theorem \ref{CongThm}, we may prove an analogous result where in  the differential equation \eqref{CapDE} we use the left Riemann--Liouville fractional derivative and the initial conditions are given by $I_{a^+}^{1-\alpha}[x](a)=x_a$ (again we refer the reader to Section \ref{sect2} to understand the meaning of the symbols used in this section). Then, we may prove the following result:

\begin{theorem}\label{thm1.4}
Let $0<\alpha< 1$ and $g\in C((a,\infty),\mathbb{R})$. Then the solution of the IVP 
\begin{equation}\label{IVPR-L}
D_{a^+}^{\alpha} [x](t)=g(t)x(t),\quad I_{a^+}^{1-\alpha}[x](a)=x_a>0,
\end{equation}
is positive on $(a,\infty)$.
\end{theorem}
However, in this work, we will be especially interested in the analogous theorem to Theorem \ref{thm1.4} but using the right fractional derivative (cf. the proof of Theorem \ref{mainHerglFract}). In order to accomplish it we may appeal to the duality results (for left and right fractional operators) presented and proved in \cite{Caputo}. We, therefore, have:

\begin{theorem}\label{thm1.5}
Let $0<\alpha< 1$ and $g\in C((-\infty,b),\mathbb{R})$. Then the solution of the IVP 
\begin{equation}\label{IVPR-L}
D_{b^-}^{\alpha} [x](t)=g(t)x(t),\quad I_{b^-}^{1-\alpha}[x](b)=x_b>0,
\end{equation}
is positive on $(-\infty,b)$.
\end{theorem}

\begin{remark}\label{rem1.6}
It is clear that, if $x_b<0$, then the solution in Theorem \ref{thm1.5} is negative on $(-\infty,b)$.
\end{remark}

The remaining of this manuscript is organized as follows: In Section \ref{sect2} we provide the reader with the definitions and results of fractional calculus needed in this work. In Section \ref{ApplicSect} we present interesting applications of Theorems \ref{thm1.3} and \ref{thm1.5} enunciated above.

\section{Preliminaries on fractional calculus}\label{sect2}

Let $I$ be an interval of $\mathbb{R}$ and $n\in\mathbb{N}$. Suppose that $E(I, \mathbb{R}^n)$ is a space of functions. We denote by $E_{loc}(I, \mathbb{R}^n)$ the space of functions $x : I\to \mathbb{R}^n$ such that $x \in E(J, \mathbb{R}^n)$ for every compact subinterval $J\subset I$. 

We now introduce the left and right fractional integrals and derivatives used in this work.

\begin{definition}Let $a<b$ be two real numbers and $[a,b]\subset I$. The left and right Riemann--Liouville fractional integrals of order $\alpha>0$ of a function $f\in L_{loc}^1(I,\mathbb{R}^n)$ are defined, respectively by
$$I_{a^+}^\alpha [f](t)=\frac{1}{\Gamma(\alpha)}\int_a^t(t-s)^{\alpha-1}f(s)ds,\ t\geq a,$$
and
$$I_{b^-}^\alpha [f](t)=\frac{1}{\Gamma(\alpha)}\int_t^b(s-t)^{\alpha-1}f(s)ds,\ t\leq b,$$
provided that the right-hand side exists.
For $\alpha=0$ we set $I_{a^+}^0[f](t)=I_{b^-}^0 [f](t)=f(t)$.
\end{definition}

\begin{definition}
The left and right Riemann--Liouville (RL) fractional derivatives of order $0<\alpha\leq 1$ of a function $f\in L_{loc}^1(I,\mathbb{R}^n)$ such that $I_{a^+}^{1-\alpha}[f]\in AC_{loc}(I,\mathbb{R}^n)$, respectively $I_{b^-}^{1-\alpha}[f]\in AC_{loc}(I,\mathbb{R}^n)$, are defined by
$$D_{a^+}^\alpha [f](t)=\frac{d}{dt}\left[I_{a^+}^{1-\alpha}[f]\right](t),$$
respectively,
$$D_{b^-}^\alpha [f](t)=-\frac{d}{dt}\left[I_{b^-}^{1-\alpha}[f]\right](t).$$
We denote by $AC_{a^+}^{\alpha}(I,\mathbb{R}^n)$, respectively $AC_{b^-}^{\alpha}(I,\mathbb{R}^n)$, the set of all functions $f\in L_{loc}^1(I,\mathbb{R}^n)$ possessing a left, respectively right, RL fractional derivative of order $0<\alpha\leq 1$.
\end{definition}

\begin{definition}
The left and right Caputo fractional derivatives of order $0<\alpha\leq 1$ of a function $f\in C(I,\mathbb{R}^n)$ such that $f-f(a)\in AC_{a^+}^{\alpha}(I,\mathbb{R}^n)$, respectively $f-f(b)\in AC_{b^-}^{\alpha}(I,\mathbb{R}^n)$, are defined by
$${^CD}_{a^+}^\alpha [f](t)={D}_{a^+}^\alpha [f-f(a)](t),$$
respectively,
$${^CD}_{b^-}^\alpha [f](t)={D}_{b^-}^\alpha [f-f(b)](t).$$
We denote by ${^CAC}_{a^+}^{\alpha}(I,\mathbb{R}^n)$, respectively ${^CAC}_{b^-}^{\alpha}(I,\mathbb{R}^n)$, the set of all functions $f\in C(I,\mathbb{R}^n)$ possessing a left, respectively right, Caputo fractional derivative of order $0<\alpha\leq 1$.
\end{definition}

We introduce the two-parametric Mittag--Leffler function
$$E_{\alpha,\beta}(z)=\sum_{k=0}^\infty\frac{z^k}{\Gamma(\alpha k+\beta)},\ \alpha>0,\ \beta\in\mathbb{C},$$
and we put $E_\alpha(z)=E_{\alpha,1}(z)$. It satisfies (cf. \cite{MS}),
\begin{equation}\label{ineqM-L}
    E_{\alpha,\beta}(t)\geq 0\mbox{ for } 0<\alpha\leq 1,\ \beta\geq\alpha.
\end{equation}

The following formulas may be found in \cite{Gorenflo}. 
\begin{lemma}
    Suppose that $\alpha,\beta,\gamma$ are positive real numbers and $\lambda\in\mathbb{R}$. Then,
    \begin{equation}\label{form2.2}
        I_{0^+}^\alpha[s^{\gamma-1}E_{\beta,\gamma}(\lambda s^\beta)](t)=t^{\alpha+\gamma-1}E_{\beta,\alpha+\gamma}(\lambda t^\beta),\quad t\geq 0,
    \end{equation}
    and 
    \begin{equation}\label{form2.3}
        E_{\alpha,\beta}(t)=\frac{1}{\Gamma(\beta)}+tE_{\alpha,\alpha+\beta}(t),\quad t\in\mathbb{R}.
    \end{equation}
\end{lemma}

The following result may be consulted in, e.g., \cite[Theorem 5.15]{Kilbas}.

\begin{theorem}[Variation of constants formula]\label{VCF}
The solution of the IVP
\begin{align}
    ^CD_{a^+}^{\alpha} [x](t)&=\lambda x(t)+f(t),\quad \lambda\in\mathbb{R},\ 0<\alpha\leq 1,\ t\geq a,\\
    x(a)&=x_a,
\end{align}
can be given by
\begin{equation}
    x(t)=x_aE_\alpha(\lambda (t-a)^\alpha)+\int_a^t(t-s)^{\alpha-1}E_{\alpha,\alpha}(\lambda (t-s)^\alpha)f(s)ds.
\end{equation}
\end{theorem}
We will also use the following result:
\begin{theorem}\label{solLinDir}
Let $f:(-\infty,b]\to\mathbb{R}$ be a continuous function. Then, the solution $x\in C((-\infty,b),\mathbb{R})$ of the IVP
\begin{align*}
    D_{b^-}^{\alpha} [x](t)&=f(t) x(t),\quad \ 0<\alpha\leq 1,\ t<b,\\
    I_{b^-}^{1-\alpha}[x](b)&=x_b,
\end{align*}
can be given by
\begin{equation}
    x(t)=\frac{x_b}{\Gamma(\alpha)}\sum_{k=0}^\infty T^k_f[(b-s)^{\alpha-1}](t),\quad t<b,
\end{equation}
where $T_f^0[\phi]=\phi$ and $T_f^{k+1}[\phi]=T_f[T_f^k\phi]$ $(k\in\mathbb{N})$, with $T_f$ being the operator defined by $T_f[\phi]=I_{b^-}^{\alpha}[f\phi]$. 
\end{theorem}

\begin{proof}
We only need to invoke the duality results of \cite{Caputo} and apply them to \cite[Theorem 2.3]{Denton}.
\end{proof}

\section{Applications}\label{ApplicSect}

\subsection{Two direct consequences of Theorem \ref{thm1.3} and a related result}

We start by stating a generalization of Theorem \ref{thm1.3}.
\begin{theorem}\label{thm3.77}
Let $0<\alpha< 1$ and $f\in C([a,\infty)\times\mathbb{R},\mathbb{R})$ be such that $f(t,0)=0$ and it satisfies the Lipschitz condition \eqref{Lip}. Then the solution of the IVP 
\begin{equation*}
    ^CD_{a^+}^{\alpha} [x](t)=f(t,x(t)),\ x(a)=x_a>0,
\end{equation*}
is positive on $[a,\infty)$.
\end{theorem}

\begin{proof}
Just observe that the trivial solution $x(t)=0$ solves the IVP $^CD_a^{\alpha} [x](t)=f(t,x(t)),\ x(a)=0$ on $[a,\infty)$ and apply Theorem \ref{CongThm}.
\end{proof}

\begin{example}
Consider $f(t,x)=g(t)\ln(x^2+1)$  with $g\in C(\mathbb{R}_0^+,\mathbb{R})$ and $x\in\mathbb{R}$.

We have,
$$|f(t,x)-f(t,y)|=|g(t)||\ln(x^2+1)-\ln(y^2+1)|\leq|g(t)||x-y|,\quad x,y\in\mathbb{R},$$
where we have used the mean value theorem. Since $f(t,0)=0$, it follows from Theorem \ref{thm3.77} that the solution of
\begin{equation*}
    ^CD_{0^+}^{\alpha} [x](t)=g(t)\ln(x^2(t)+1),\ x(0)=1,\ t\geq 0,
\end{equation*}
is positive.
\end{example}
The following result seems to be new in the literature.
\begin{theorem}
Let $x_a>0$. For $g\in C([a,\infty],\mathbb{R})$ and $0<\alpha<1$ define
$$k(t,s)=\frac{1}{\Gamma{(\alpha)}}(t-s)^{\alpha-1}g(s),$$
and  the $j$th iterated kernel $k_j$ for $j=1,2,\ldots $ via the recurrence relation
$$k_1(t,s)=k(t,s),\quad k_j(t,s)=\int_s^tk(t,\tau)k_{j-1}(\tau,s)d\tau,\ j=2,3,\ldots .$$
Then the function
$$x(t)=x_a\left(1+\int_a^tR(t,s)ds\right),\quad t\in[a,\infty),$$
where $R(t,s)=\sum_{j=1}^\infty k_j(t,s)$ is positive.
\end{theorem}

\begin{proof}
This result follows from the representation for the solution of \eqref{IVPCap} given in \cite[Theorem 7.10]{Diethelm}.
\end{proof}

\begin{remark}
Observe that the previous result is by no means obvious as the function $g$ may be negative.
\end{remark}

We end this section with a lateral but nevertheless interesting result, namely, we deduce a Bernoulli-type inequality using the theory of fractional calculus\footnote{We have never seen such type of result in the literature.}. The proof follows the same lines as the one in \cite{Ferreira1} (see also \cite{Ferreira2}), which was done using discrete fractional operators.

\begin{theorem}[Fractional Bernoulli's inequality]\label{ThmBern}
Let $\lambda\in\mathbb{R}$ and $0<\alpha\leq 1$. Then the following inequality holds:
\begin{equation}\label{BernIneq}
    E_{\alpha}(\lambda t^\alpha)\geq\frac{\lambda t^\alpha}{\Gamma(\alpha+1)}+1,\quad t\geq 0.
\end{equation}
\end{theorem}

\begin{proof}
The formula trivially holds for $t=0$.
Let 
$$x(t)=\lambda\frac{t^\alpha}{\Gamma(\alpha+1)},\quad t\geq 0.$$
We have $x(0)=0$ and $^CD_{0^+}^{\alpha} x(t)=\lambda$ (cf. \cite[Property 2.1]{Kilbas}). Therefore,
$$\lambda x(t)+\lambda=\lambda^2\frac{t^\alpha}{\Gamma(\alpha+1)}+\lambda\geq {^CD}_0^{\alpha} x(t).$$
Define $m(t)=\lambda x(t)+\lambda-{^CD}_{0^+}^{\alpha} x(t)$, which is nonnegative. Then, using Theorem \ref{VCF}, we get
\begin{align*}
    x(t)&=\int_0^t(t-s)^{\alpha-1}E_{\alpha,\alpha}(\lambda (t-s)^\alpha)[\lambda-m(s)]ds\\
    &=\lambda\int_0^t(t-s)^{\alpha-1}E_{\alpha,\alpha}(\lambda (t-s)^\alpha)ds-\int_0^t(t-s)^{\alpha-1}E_{\alpha,\alpha}(\lambda (t-s)^\alpha)m(s)ds\\
    &\leq \lambda\int_0^t(t-s)^{\alpha-1}E_{\alpha,\alpha}(\lambda (t-s)^\alpha)ds\\
    &=\lambda t^\alpha E_{\alpha,\alpha+1}(\lambda t^\alpha),
\end{align*}
where we have used \eqref{ineqM-L} and \eqref{form2.2}. Therefore,
$$\lambda\frac{t^\alpha}{\Gamma(\alpha+1)}\leq \lambda t^\alpha E_{\alpha,\alpha+1}(\lambda t^\alpha),$$
which for $t>0$ is equivalent to
$$\frac{\lambda}{\Gamma(\alpha+1)}\leq \lambda E_{\alpha,\alpha+1}(\lambda t^\alpha),$$
and upon using \eqref{form2.3} and some rearrangements furnishes \eqref{BernIneq}.
\end{proof}

\begin{remark}
It is clear that, for $\alpha=1$, inequality \eqref{BernIneq} reads as $e^{\lambda t}\geq \lambda t+1$. This inequality is the continuous version of the Bernoulli inequality\footnote{That is, $(1+x)^n\geq 1+nx$ for $x>-1$ and $n\in\mathbb{N}$.} (cf. \cite{Agarwal}), hence the name given in Theorem \ref{ThmBern}.
\end{remark}

\subsection{Herglotz's variational problem}

In this section we will deduce necessary optimality conditions for a fractional variational problem of Herglotz type \cite{Georgieva,Herglotz} and, in particular, show the usefulness of Theorem \ref{thm1.5}. Let us first state what we mean here by the Herglotz variational problem in the classical case:  Following \cite[Problem $P_H$]{Santos}, we consider:
\begin{align}
    z(b)&\longrightarrow \min\nonumber\\
    \mbox{subject to}\ \dot z(t)&=L(t,x(t),\dot x(t),z(t)),\ t\in[a,b],\label{ClassHerglotzprobl}\\
     x(a)=&x_a,\ z(a)=z_a,\quad x_a,z_a\in\mathbb{R}.\nonumber
\end{align}
The aim is to find a couple $(x,z)$ in an appropriate space of functions that solves \eqref{ClassHerglotzprobl}, i.e. that satisfy all the conditions in \eqref{ClassHerglotzprobl}.

In our work we will consider the following fractional version of \eqref{ClassHerglotzprobl}:
\begin{align}
    z(b)&\longrightarrow \min\nonumber\\
    \mbox{subject to}\ {^CD}_{a^+}^\alpha[z](t)&=L(t,x(t),{^CD}_{a^+}^\alpha[x](t),z(t)),\ t\in[a,b]\label{FracHerglotzprobl},\\
     x(a)=&x_a,\ z(a)=z_a,\quad x_a,z_a\in\mathbb{R}.\nonumber
\end{align}
Before proceeding let us just note that, if the function $L$ does not depend on its fourth variable, then $z$ is immediately determined and we may write \eqref{FracHerglotzprobl} as
\begin{align*}
    \frac{1}{\Gamma(\alpha)}\int_a^b(b-t)^{\alpha-1}&\left[ L(t,x(t),{^CD}_{a^+}^\alpha[x](t))+\frac{\Gamma(\alpha)z_a(b-t)^{1-\alpha}}{b-a}\right]dt\longrightarrow \min\\
    &\mbox{subject to}\
     x(a)=x_a,\quad x_a\in\mathbb{R},
\end{align*}
in view of $I_{a+}^{\alpha}[{^CD}_{a^+}^\alpha[z]](t)=z(t)-z(a)$. So, if we define $$\hat{L}(t,x,v)=L(t,x,v)+\frac{\Gamma(\alpha)z_a(b-t)^{1-\alpha}}{b-a},$$
we get the problem
\begin{align*}
    \mathcal{L}(x)=\frac{1}{\Gamma(\alpha)}\int_a^b(b-t)^{\alpha-1}& \hat{L}(t,x(t),{^CD}_{a^+}^\alpha[x](t))dt\longrightarrow \min\\
    &\mbox{subject to}\
     x(a)=x_a,\quad x_a\in\mathbb{R}.
\end{align*}
This is the basic problem of the fractional calculus of variations with fixed initial condition. The same problem with fixed initial and final conditions was recently studied in \cite{Ferreira}. 

To the best of our knowledge the first work considering Herglotz-type problems involving fractional derivatives is \cite{Almeida}. However, the problem we consider here is different in nature from the one considered in \cite{Almeida} (when $0<\alpha<1$) as the authors considered the differential equation
$$\dot z(t)=L(t,x(t),{^CD}_{a^+}^\alpha[x](t),z(t)),$$
instead of ${^CD}_{a^+}^\alpha[z](t)=L(t,x(t),{^CD}_{a^+}^\alpha[x](t),z(t))$ considered above in \eqref{FracHerglotzprobl}. Moreover, the proofs of the necessary optimality conditions are quite different.

In this work we will obtain first and second order necessary otimality conditions for \eqref{FracHerglotzprobl} by using a recent result of \cite{Bourdin}, namely, the fractional version of the celebrated Pontryagin Maximum Principle (PMP). For the benefit of the reader we recall here the main result of \cite{Bourdin}.

Consider the Optimal Control Problem (OCP) of Bolza type given by\footnote{In the notation of \cite{Bourdin} we considering here $\beta=\alpha$.}
\begin{align*}
    \varphi(x(a),x(b))+I_{a^+}^\alpha&[F(\cdot,x,u)](b)\longrightarrow \min\\
    \mbox{subject to}\ x\in {^CA}C_{a^+}^{\alpha}&([a,b],\mathbb{R}^n),\ u\in L^{\infty}([a,b],\mathbb{R}^m),\\
    {^CD}_{a^+}^\alpha[x](t)&=f(t,x(t),u(t))\ a.e.\ t\in[a,b],\\
     g(x(a)&,x(b))\in C,\\
     u(t)\in U&\ a.e.\ t\in[a,b].
\end{align*}
A couple $(x^\ast, u^\ast)$ is said to be an optimal solution to OCP if it satisfies all the above constraints and it minimizes the cost among all couples $(x, u)$ satisfying those constraints. Obviously, the functions involved in the OCP satisfy some regularity conditions, that we will skip here and refer the reader to \cite[Page 8]{Bourdin}. Under these hypothesis (regularity conditions), we have the following theorem.

\begin{theorem}[PMP]\label{ThmPMP}
Assume that $(x^\ast,u^\ast)\in {^CA}C_{a^+}^{\alpha}([a,b],\mathbb{R}^n)\times L^{\infty}([a,b],\mathbb{R}^m)$ is an optimal solution to the OCP. Then, there exists a nontrivial couple $(p,p^0)$, where $p\in AC_{b^-}^{\alpha}([a,b],\mathbb{R}^n)$ (called adjoint vector) and $p^0\leq 0$, such that the following conditions hold:
\begin{enumerate}[label=(\roman*)]
\item Fractional Hamiltonian system:
\begin{align*}
    {^CD}_{a^+}^\alpha[x^\ast](t)&=\partial_4 H(t,x^\ast(t),u^\ast(t),p(t),p^0)\\
    D_{b^-}^\alpha[p](t)&=\partial_2 H(t,x^\ast(t),u^\ast(t),p(t),p^0),
\end{align*}
for almost every $t\in[a,b]$, where the Hamiltonian $H:[a,b)\times\mathbb{R}^n\times\mathbb{R}^m\times\mathbb{R}^n\times\mathbb{R}\to\mathbb{R}$ associated to
Problem (OCP) is defined by
$$H(t,x,u,p,p^0)=\langle p,f(t,x,u)\rangle_{\mathbb{R}^n}+p^0\frac{(b-t)^{\alpha-1}}{\Gamma(\alpha)}F(t,x,u).$$
\item Hamiltonian maximization condition:
$$u^\ast(t)=\arg\max_{u\in U} H(t,x^\ast(t),u,p(t),p^0)\ a.e.\ t\in[a,b].$$
\item Transversality conditions on the adjoint vector: if in addition $g$ is submersive\footnote{A function $g : \mathbb{R}^n\times\mathbb{R}^n \to \mathbb{R}^j$ is said to be submersive at a point $(x_a,x_b)\in\mathbb{R}^n\times\mathbb{R}^n$ if its differential at this point is surjective.} at $(x^\ast(a),x^\ast(b)$, then the couple $(p,p^0)$ satisfy
\begin{align*}
    I_{b^-}^{1-\alpha}[p](a)&=-p^0\partial_1\varphi(x^\ast(a),x^\ast(b))-\partial_1 g(x^\ast(a),x^\ast(b))^T\times\Psi,\\
    I_{b^-}^{1-\alpha}[p](b)&=p^0\partial_2\varphi(x^\ast(a),x^\ast(b))+\partial_2 g(x^\ast(a),x^\ast(b))^T\times\Psi,
\end{align*}
where $\Psi\in\mathcal{N}_C[g(x^\ast(a),x^\ast(b))]$, with $\mathcal{N}_C[x]=\{z\in\mathbb{R}^j:\forall x'\in C, \langle z,x'-x\rangle_{\mathbb{R}^j}\}\leq 0\}$.
\end{enumerate}
\end{theorem}

A series of remarks is in order.

\begin{remark}\label{rem3.2}
If $U=\mathbb{R}$, i.e., there is no control constraint in the OCP, and the Hamiltonian is differentiable with respect to its third variable, then the maximization condition (ii) in Theorem \ref{ThmPMP} implies (cf. \cite[Remark 3.18]{Bourdin})
\begin{equation*}
     \partial_3H(t,x^\ast(t),u^\ast(t),p(t),p^0)=0\ a.e.\ t\in[a,b].
\end{equation*}
Moreover, if $H$ is twice differentiable with respect to its third variable, we easilly see that
 \begin{equation} \label{ineqSo} 
    \partial_{33}H(t,x^\ast(t),u^\ast(t),p(t),p^0)\leq 0\ a.e.\ t\in[a,b].
\end{equation}
\end{remark}

\begin{remark}\label{rem3.3}
If the initial point is fixed and if the final point is free in the OCP, then we may take (cf. \cite[Remark 3.17]{Bourdin})
\begin{equation*}
     I_{b^-}^{1-\alpha}[p](b)=-\partial_2\varphi(x^\ast(a),x^\ast(b)).
\end{equation*}
\end{remark}

\begin{remark}\label{rem3.4}
It is mentioned in \cite[pag. 15]{Bourdin} and shown in \cite[Theorem 5.3]{Bourdin1} that $p\in (C[a,b),\mathbb{R}^n)$.
\end{remark}

It follows the main result of this section:

\begin{theorem}\label{mainHerglFract}
Consider the function $L(t,x,u,z)$ in \eqref{FracHerglotzprobl} to have continuous partial derivatives with respect to $x,\ u$ and $z$. Suppose that  $(x^\star,z^\star)\in{^CA}C_{a^+}^{\alpha}([a,b],\mathbb{R})\times {^CA}C_{a^+}^{\alpha}([a,b],\mathbb{R})$, with ${^CD}_{a^+}^\alpha[x^\star]\in C([a,b],\mathbb{R})$, solves the Herglotz problem \eqref{FracHerglotzprobl}. Then
\begin{equation}\label{FOCthm}
    I_{b^-}^{\alpha}[p \partial_2 L(\cdot,x^{\star},{^CD}_{a^+}^\alpha[x^\star],z^{\star})](t)+p(t)\partial_3 L(t,x^{\star}(t),{^CD}_{a^+}^\alpha[x^\star](t),z^{\star}(t))=0,
\end{equation}
for all $t\in[a,b)$, 
where $p$ is the solution of
$$D_{b^-}^\alpha[p](t)=p(t)\partial_4L(t,x^{\star}(t),{^CD}_{a^+}^\alpha[x^\star](t),z^{\star}(t)),\ t\in[a,b),\ I_{b^-}^{1-\alpha}[p](b)=-1.$$
Moreover, if $L$ is twice continuously differentiable with respect to $u$, then the Legendre necessary optimality condition holds:
\begin{equation}\label{Legendre}
\partial_{33}L(t,x^{\star}(t),{^CD}_{a^+}^\alpha[x^\star](t),z^{\star}(t))\geq 0,\ t\in[a,b].
\end{equation}
\end{theorem}

\begin{proof}
Suppose that $(x^\star,z^\star)$ is a solution of \eqref{FracHerglotzprobl}. Then, by letting $x^{1\star}(t)=x^\star(t)$, $x^{2\star}(t)=z^\star(t)$ and $u^\star(t)={^CD}_{a^+}^\alpha[x^\star](t)$,  we conclude that $( x^{1\star}, x^{2\star},u^\star)$ solves the following OCP
\begin{align*}
    & x^{2}(b)\longrightarrow \min\\
    \mbox{subject to}&\ {^CD}_{a^+}^\alpha[x^{1}](t)=u(t),\ t\in[a,b],\\
    &\ {^CD}_{a^+}^\alpha[x^{2}](t)=L(t,x^{1}(t),u(t),x^{2}(t)),\ t\in[a,b],\\
     &\ x^{1}(a)=x_a,\ x^{2}(a)=z_a,\quad x_a,z_a\in\mathbb{R}.
\end{align*}
It follows from Theorem \ref{ThmPMP} and Remarks \ref{rem3.2}, \ref{rem3.3} and \ref{rem3.4} the existence of a vector $(p_1,p_2)\in C([a,b),\mathbb{R}^2)$ satisfying
\begin{equation}\label{FOC}
    p_1(t)+p_2(t)\partial_3L(t,x^{1\star}(t),u^\star(t),x^{2\star}(t))=0,\ a.e.\ t\in[a,b],
\end{equation}
and 
\begin{align}
     D_{b^-}^\alpha[p_1](t)&=p_2(t)\partial_2L(t,x^{1\star}(t),u^\star(t),x^{2\star}(t)),\ a.e.\ t\in[a,b],\ I_{b^-}^{1-\alpha}[p_1](b)=0,\label{eqqq0}\\
     D_{b^-}^\alpha[p_2](t)&=p_2(t)\partial_4L(t,x^{1\star}(t),u^\star(t),x^{2\star}(t)),\ a.e.\ t\in[a,b],\ I_{b^-}^{1-\alpha}[p_2](b)=-1.\label{eqqq1}
\end{align}
Observe that the continuity of $p_1$ and $p_2$ on $[a,b)$, together with the assumptions on $L$ and $(x^\star,u^\star)$, imply that \eqref{FOC}, \eqref{eqqq0} and \eqref{eqqq1} hold on $[a,b)$. Moreover, since $p_1(t)=I_{b^-}^{\alpha}[p_2\partial_2L(\cdot,x^{1\star},u^\star,x^{2\star})](t)$, then \eqref{FOCthm} follows from immediately \eqref{FOC}.

Suppose now that $L$ is  twice  continuously  differentiable  with  respect  to $u$. Then, by \eqref{ineqSo}, we get
 \begin{equation}\label{almLegen}
    p_2(t)\partial_{33}L(t,x^{1\star}(t),u^\star(t),x^{2\star}(t))\leq 0,\ a.e.\ t\in[a,b],
\end{equation}
and, upon using Theorem \ref{thm1.5}, Remark \ref{rem1.6} and the continuity of $p_2$ on $[a,b)$,
$$\partial_{33}L(t,x^{1\star}(t),u^\star(t),x^{2\star}(t))\geq 0,\ a.e.\ t\in[a,b].$$
The previous inequality holds on $[a,b]$ from the hypothesis on $L$ and $(x^\star,u^\star)$. The proof is done.
\end{proof}

We call \eqref{FOCthm} the \emph{Euler--Lagrange equation in integral form} for the Herglotz variational problem \eqref{FracHerglotzprobl}.

\begin{remark}
The function $p$ of Theorem \ref{mainHerglFract} has the representation (cf. Theorem \ref{solLinDir}):
$$p(t)=-\frac{1}{\Gamma(\alpha)}\sum_{k=0}^\infty T^k_f[(b-s)^{\alpha-1}](t),\quad t<b,$$
where $f(t)=\partial_4L(t,x^{\star}(t),{^CD}_{a^+}^\alpha[x^\star](t),z^{\star}(t))$.
\end{remark}

\begin{remark}
We emphasize the importance of Theorem \ref{thm1.5} in order to obtain the Legendre necessary condition \eqref{Legendre}. Because of it we were able to remove the dependence on the function $p_2$ in the inequality \eqref{almLegen}.

Also, suppose that $L(t,x,u,z)$ does not depend on $x$, i.e., $L(t,x,u,z)=\hat{L}(t,u,z)$. Then, the Euler--Lagrange equation \eqref{FOCthm} becomes
$$p(t)\partial_2 \hat{L}(t,{^CD}_{a^+}^\alpha[x^\star](t),z^{\star}(t))=0,\quad t\in[a,b).$$
Again, we may use Theorem \ref{thm1.5} and the continuity of the involved functions to conclude that
$$\partial_2 \hat{L}(t,{^CD}_{a^+}^\alpha[x^\star](t),z^{\star}(t))=0,\quad t\in[a,b].$$
\end{remark}
If we let $\alpha=1$ in Theorem \ref{mainHerglFract}, it follows the following:
\begin{corollary}
The first and second order optimality conditions for the variational problem given by \eqref{ClassHerglotzprobl} are, respectively,
\begin{equation}\label{alpha1}
    \int_t^be^{\int_s^b\partial_4L[\tau]d\tau} \partial_2 L[s]ds+e^{\int_t^b\partial_4L[s]ds}\partial_3 L[t]=0,\ t\in[a,b],
\end{equation}
where $[s]=(s,x^{\star}(s),\dot x^\star(s),z^{\star}(s))$,
and 
\begin{equation*}
\partial_{33}L(t,x^{\star}(t),\dot x^\ast(t),z^{\star}(t))\geq 0,\ t\in[a,b].
\end{equation*}
\end{corollary}

\begin{proof}
Just let $\alpha=1$ in Theorem \ref{mainHerglFract} and note that, in this case, $p(t)=-e^{\int_t^b\partial_4L[s]ds}$, for all $t\in[a,b]$. 
\end{proof}
We end this work by noting that we can obtain a differential form for equation \eqref{alpha1}. Indeed, since the integral on the left hand side of \eqref{alpha1} and $f(t)=e^{\int_t^b\partial_4L[s]ds}>0$ are differentiable on $[a,b]$, then $\partial_3 L$ is also differentiable and, hence, we easilly obtain
\begin{multline*}
\partial_2 L(t,x^{\star}(t),\dot x^\star(t),z^{\star}(t))+\partial_4 L(t,x^{\star}(t),\dot x^\star(t),z^{\star}(t))\partial_3 L(t,x^{\star}(t),\dot x^\star(t),z^{\star}(t))\\
-\frac{d}{dt}\partial_3 L(t,x^{\star}(t),\dot x^\star(t),z^{\star}(t))=0,\quad t\in[a,b].
\end{multline*}

\section*{Acknowledgments}
The author would like to thank the referees for their careful reading of the manuscript, and their corrections and suggestions which contributed to improve this article.

\bibliographystyle{amsplain}

\end{document}